\documentclass{article}
\usepackage[margin=1.4in]{geometry}
\usepackage{amsmath} 
\usepackage{amsfonts} 
\usepackage{amsthm} 
\usepackage{tikz-cd} 
\usepackage[mathscr]{eucal}
\usepackage{hyperref}
\hypersetup{colorlinks}
\usepackage{cleveref} 

\newtheorem{THM}{Theorem}
\newtheorem{thm}{Theorem}
\numberwithin{thm}{section}
\newtheorem{lem}[thm]{Lemma}
\newtheorem{prop}[thm]{Proposition}

\DeclareMathOperator{\out}{Out}
\DeclareMathOperator{\aut}{Aut}
\DeclareMathOperator{\inn}{Inn}
\DeclareMathOperator{\vol}{vol}
\DeclareMathOperator{\lip}{Lip}
\newcommand{\PT}{\mathscr{PT}}
\newcommand{\doublebackslash}{\backslash\mkern-5mu\backslash}

\begin{document}
\title{Lipschitz metric isometries between Outer Spaces of virtually free groups}
\author{Rylee Alanza Lyman}
\maketitle
\begin{abstract}
Dowdall and Taylor observed that
given a finite-index subgroup of a free group,
taking covers induces an embedding
from the Outer Space of the free group
to the Outer Space of the subgroup,
that this embedding
is an isometry with respect to the (asymmetric) Lipschitz metric,
and that the embedding sends folding paths to folding paths.
The purpose of this note is to extend this result to virtually free groups.
We further extend a result Francaviglia and Martino,
proving the existence of ``candidates'' for the Lipschitz distance between points 
in the Outer Space of the virtually free group.
Additionally we identify a deformation retraction of the spine of the Outer Space
for the virtually free group
with the space considered by Krsti\'c and Vogtmann.
\end{abstract}

Let $F$ be a finitely generated virtually free group.
Throughout this paper we will always assume $F$ is virtually \emph{non-abelian} free,
but will just write \emph{virtually free.}
Consider the \emph{deformation space} $\mathscr{T} = \mathscr{T}(F)$ in the sense of
Forester and Guirardel--Levitt~\cite{GuirardelLevittDeformation,Forester}
whose points are actions of $F$ on simplicial $\mathbb{R}$-trees with finite stabilizers.
This deformation space is \emph{canonical} in the sense that all of $\out(F)$ acts on it,
and we call it the \emph{unprojectivized Outer Space} of $F$.
Projectivizing (by for example requiring the volume of the quotient graph of groups to be $1$)
yields the \emph{Outer Space} $\PT = \PT(F)$.

When $F = F_n$, the space $\PT(F_n)$ is the Culler--Vogtmann Outer Space $\operatorname{CV}_n$.
When $F$ is a free product of finite and cyclic groups,
the space $\PT$ is not quite Guirardel--Levitt's Outer Space $\mathscr{O}(F)$ 
for $F$~\cite{GuirardelLevitt}:
the difference is that trees in $\mathscr{O}(F)$ have trivial edge stabilizers,
while there is no such requirement for trees in $\PT(F)$.

In~\cite[Section 5.1]{DowdallTaylor}, Dowdall--Taylor
show that if $H$ is a finite-index subgroup of $F_n$,
then thinking of each $F_n$-tree $T \in \PT(F_n)$ as an $H$-tree
yields an embedding $i\colon \PT(F_n) \to \PT(H)$.
They show that $i$ is an isometry with respect to the (asymmetric) \emph{Lipschitz metric}
on $\PT(F_n)$ and $\PT(H)$,
and that $i$ sends a particular class of (directed) geodesics in $\PT(F_n)$
called \emph{folding paths} to folding paths in $\PT(H)$.
The purpose of this note is to extend these results to virtually free groups.

\begin{THM}\label{isometricembeddingthm}
    Let $H \le F$ be a finite-index subgroup of the virtually free group $F$.
    The induced map $i\colon \PT(F) \to \PT(H)$ is an isometry with respect to the Lipschitz metric.
    Moreover, $i$ maps folding paths in $\PT(F)$ to folding paths in $\PT(H)$.
\end{THM}

The Lipschitz metric for $\PT(F)$ when $F$ is not free
was first defined and studied by Meinert~\cite{Meinert}.
He identifies $\PT$ with the subset of $\mathscr{T}$
consisting of those trees $\Gamma$
such that the sum of the lengths of edges of the quotient graph $F\backslash\Gamma$ is $1$.
This notion of volume is not so well-behaved for general $\PT$:
the Lipschitz metric is only a pseudometric,
and taking covers is not multiplicative with respect to volume.
We introduce a new notion of volume which rectifies both of these issues:
we divide the length of an edge by the order of its stabilizer and then add.

Meinert proves~\cite[Theorem 4.14]{Meinert} that there are \emph{witnesses}
for the Lipschitz pseudometric distance between any two trees in $\mathscr{T}$:
hyperbolic group elements $g \in F$ such that
\[  d(\Gamma,\Gamma') = \log \frac{\|g\|_{\Gamma'}}{\|g\|_\Gamma},  \]
where $\|g\|_\Gamma$ denotes the hyperbolic translation length of $g$ on $\Gamma$.
He further shows that these witnesses may be taken to satisfy a weak finiteness condition.
By contrast, Francaviglia--Martino~\cite{FrancavigliaMartino,FrancavigliaMartinoTrainTracks}
show that in the case of $F$ a free product,
the projection of $g$ to the quotient graph of groups has certain nice properties.
To reduce the computational complexity of computing the Lipschitz distance 
between trees in $\PT$, we extend Francaviglia--Martino's result to $F$ virtually free.

\begin{THM}[cf.\ Theorem 9.10 of~\cite{FrancavigliaMartinoTrainTracks}]\label{candidatesexist}
    Let $\Gamma$ and $\Gamma'$ be trees in $\mathscr{T}$.
    There is a witness $g \in F$,
    the projection of whose axis into the quotient graph of groups $F\doublebackslash\Gamma$
    has one of the following forms.
    \begin{enumerate}
        \item An embedded simple loop.
        \item An embedded figure-eight (a wedge of two circles).
        \item An embedded barbell (two simple loops joined by a segment).
        \item An embedded singly degenerate barbell 
            (a non-free vertex and a simple loop joined by a segment).
        \item An embedded doubly degenerate barbell
            (two non-free vertices joined by a segment).
    \end{enumerate}
\end{THM}

There are finitely many graph-of-groups edge paths
of the given forms, which we call \emph{candidates,}
and the set of candidates is independent of the choice of $\Gamma'$.

\paragraph{}
In~\cite{BestvinaFeighnHandel} and~\cite{KrsticVogtmann},
given a finite-rank non-abelian free group $F_n$ and a homomorphism
$\alpha\colon G \to \out(F_n)$ with finite domain,
a certain subcomplex of the \emph{spine} of Outer Space they term $L_\alpha$ is considered.
In those papers, vertices of $L_\alpha$ are marked graphs equipped with a $G$-action via $\alpha$,
and the contractibility of these simplicial complexes is used to prove that
centralizers of finite subgroups of $\out(F_n)$ are virtually of type F.

Let $G \le \out(F_n)$ be a finite subgroup, 
and let $E$ be the full preimage of $G$ in $\aut(F_n)$.
Then $E$ is a virtually free group.
Let $\aut^0(E)$ be the subgroup of $\aut(E)$ that leaves $F_n \cong \inn(F_n)$ invariant
and let $\out^0(E)$ be its image in $\out(E)$.
We define a deformation retraction $L = L(E)$ of the spine of the outer space of $E$.
Our final result is the following theorem;
see \Cref{precisethmC} for a precise statement.

\begin{THM}\label{normalizerthm}
    Given $G \le \out(F_n)$ and $E$ as above,
    the isometry $i\colon \PT(E) \to \PT(F_n)$ induces a simplicial isomorphism
    $L(E) \to L_\alpha$, where $\alpha\colon G \to \out(F_n)$ is the natural inclusion.
    The normalizer $N(G)$ of $G$ in $\out(F_n)$ leaves $L_\alpha$ invariant.
    We have the following short exact sequence
    \[  \begin{tikzcd}
        1 \ar[r] & G \ar[r] & N(G) \ar[r] & \out^0(E) \ar[r] & 1.
    \end{tikzcd}\]
\end{THM}

In particular, if $F_n$ is characteristic in $E$, then $\out^0(E) = \out(E)$.

Here is the organization of this note.
We begin by recalling work of Meinert in \Cref{Lipschitzsection} 
culminating in the proof of \Cref{isometricembeddingthm}.
\Cref{candidatessection} is dedicated to the proof of \Cref{candidatesexist},
and \Cref{normalizerthm} is proved in \Cref{spinesection}.

\section{Outer Space and the Lipschitz metric}\label{Lipschitzsection}
\paragraph{Unprojectivized Outer Space.}
A point of $\mathscr{T} = \mathscr{T}(F)$
for a finitely generated virtually free group $F$
is the equivalence class of a minimal isometric action of $F$ 
on a simplicial $\mathbb{R}$-tree $\Gamma$
such that all stabilizers are finite.
The equivalence relation is $F$-equivariant isometry of trees.
We will, as is standard, always speak of actual trees,
rather than isometry classes of trees.

The group $\aut(F)$ acts on $\mathscr{T}$ on the right
by twisting the action,
i.e.~the tree $\Gamma.\Phi$ is the tree $\Gamma$
equipped with the $F$-action defined by first applying $\Phi$ and then acting.
Inner automorphisms of $F$ correspond to $F$-equivariant isometries,
so we get an action of $\out(F)$ on $\mathscr{T}$.

\paragraph{Quotient graph of groups.}
Associated to such a tree $\Gamma$, there is a \emph{quotient graph of groups}
$\mathcal{G} = F\doublebackslash\Gamma$, which we now describe,
following~\cite[3.2]{Bass}.
By subdividing at midpoints of edges of $\Gamma$ which are inverted by the $F$-action,
we may assume that the action of $F$ on $\Gamma$ is without inversions in edges,
so the quotient graph $G = F\backslash\Gamma$ naturally inherits a graph structure.
Write $p\colon \Gamma \to G$ for the quotient map.
Choose connected subgraphs $T \subset S \subset \Gamma$
such that $p|_S\colon S \to G$ is a bijection on edges (so $S$ is a \emph{fundamental domain} for the action)
and $p|_T\colon T \to G$ is a bijection on vertices (so $T$ projects to a \emph{spanning tree}).
For $e$ an edge and $v$ a vertex of $G$,
write $\tilde e$ and $\tilde v$ for the unique preimage of $e$ and $v$
in $S$ and $T$ respectively.
For an oriented edge $e$ of a graph, write $\tau(e)$ for its terminal vertex.
For each oriented edge $e$ of $G$,
let $g_e \in F$
be an element such that
$g_e.\tau(\tilde e) = \widetilde{\tau(e)}$.
If $\tau(\tilde e) \in T$, choose $g_e = 1$.
The groups $\mathcal{G}_v$ and $\mathcal{G}_e$
for $v$ a vertex and $e$ an edge of $G$ are the stabilizers in $F$
of $\tilde v$ and $\tilde e$ respectively.
For an oriented edge $e$ of $G$,
the injective homomorphism $\mathcal{G}_e \to \mathcal{G}_{\tau(e)}$
is the restriction of the map $x \mapsto g_e x g_e^{-1}$ to $\mathcal{G}_e$.

The graph of groups $\mathcal{G} = F\doublebackslash\Gamma$
is equipped with a \emph{metric} defined so that $p\colon \Gamma \to \mathcal{G}$
restricts to an isometry on each edge
and a \emph{marking,} an isomorphism $F\cong \pi_1(\mathcal{G})$ well-defined up to inner automorphism.

\paragraph{Volume, projectivized Outer Space.}
Given a tree $\Gamma \in \mathscr{T}$, we define the \emph{volume} of $\Gamma$ to be
\[  \vol(\Gamma) = \vol(F\doublebackslash\Gamma) = \sum_{e} \frac{\ell(e)}{|\mathcal{G}_e|}, \]
where $\ell(e)$ denotes the length of the edge $e$,
and where the sum is taken over the edges $e$ of $F\doublebackslash\Gamma$
(equivalently the edges of some and hence any fundamental domain $S$).
It is clear that if $\Gamma'$ has hyperbolic length function $\|\cdot\|_{\Gamma'}$
(see e.g.~\cite{CullerMorgan})
equal to $\lambda\|\cdot\|_\Gamma$ for some $\lambda > 0$,
then $\vol(\Gamma') = \lambda\vol(\Gamma)$,
so we may and will identify $\PT(F)$ with the subset of $\mathscr{T}(F)$
comprising the trees with volume $1$.
Let us remark for experts that because $\mathcal{T}(F)$ is locally compact
(and its spine is locally finite),
this identification is as innocuous as it is for $\operatorname{CV}_n$.

\paragraph{Stretch factors.}
We begin by reporting on work of Meinert~\cite{Meinert}.
Given two trees $\Gamma$ and $\Gamma'$ in $\mathscr{T}$,
there is an $F$-equivariant map $f\colon \Gamma \to \Gamma'$,
which we may take to be Lipschitz continuous,
and we write $\lip(f)$ for its Lipschitz constant.
As in~\cite[p.~998]{Meinert}, it suffices to consider
only those $F$-equivariant maps which are \emph{piecewise linear,}
which is to say for each edge $\tilde e$ of $\Gamma$,
the map $f$ is either constant on $\tilde e$ or linear with constant slope on $\tilde e$.
We define
\[  \lip(\Gamma,\Gamma') = \inf\{\lip(f) : f\colon \Gamma \to \Gamma' \text{ is $F$-equviariant}\}.\]
Meinert proves~\cite[Theorem 4.6]{Meinert}
using nonprincipal ultrafilters
that there exists an $F$-equivariant piecewise linear Lipschitz map
$f\colon \Gamma \to \Gamma'$ such that $\lip(f) = \lip(\Gamma,\Gamma')$.
We remark that since our trees are locally finite
and the quotient graphs of groups are compact,
it should be possible to give a proof using the Arzela--Ascoli theorem.

Meinert further proves~\cite[Theorem 4.14]{Meinert} that
\[  \lip(\Gamma,\Gamma') = \sup_{g} \frac{\|g\|_{\Gamma'}}{\|g\|_{\Gamma}}, \]
where the supremum is taken over all hyperbolic elements $g \in F$,
and that the supremum is realized.

\paragraph{Train track structures.}
Given a tree $\Gamma \in \mathscr{T}$ and a point $\tilde x \in \Gamma$,
a \emph{direction} at $\tilde x$
is a germ of geodesics beginning at $\tilde x$.
Equivalently, a direction is a component of the complement $\Gamma\setminus\{\tilde x\}$.
The set of directions at a point $\tilde x$ is denoted $D_{\tilde x}\Gamma$.
The action of $F$ on $\Gamma$ induces a permutation of the set of direction in $\Gamma$.
A \emph{train track structure} on $\Gamma$
is an $F$-equivariant collection of equivalence relations on the set of directions of $\Gamma$
with the property that there are at least two equivalence classes at each point.
That is to say, if $\delta_1\sim\delta_2$ as directions in $D_{\tilde x}\Gamma$,
then for all $g \in F$, we have $g.\delta_1 \sim g.\delta_2$ as directions in $D_{g.\tilde x}\Gamma$.
A train track structure partitions the directions at $D_{\tilde x}\Gamma$
into \emph{gates,}
and an $F$-equivariant collection of equivalence relations on the set of directions of $\Gamma$
is a train track structure if and only if there are at least two gates at every point.
A \emph{turn} is a pair of directions based at a common point.
It is \emph{illegal} if the directions belong to the same gate, and \emph{legal} otherwise.
A geodesic $\tilde\gamma$ in $\Gamma$ through $\tilde x$ is said to
\emph{cross} the turn at $\tilde x$ determined by the components of the complement
$\Gamma\setminus\{\tilde x\}$ containing $\tilde\gamma$.
A geodesic $\tilde\gamma$ is \emph{legal} if it makes only legal turns and \emph{illegal} otherwise.

\paragraph{Tension forest.}
Let $f\colon \Gamma \to \Gamma'$ be a piecewise linear $F$-equivariant map
between trees in $\mathscr{T}$.
The union of all closed edges of $\Gamma$
on which $f$ attains its maximal slope (which must be nonzero)
is $\Delta(f)$, the \emph{tension forest} of $f$.
It is $F$-invariant.
The map $f$ sends directions pointing into $\Delta = \Delta(f)$ to directions of $\Gamma'$,
and thus defines an equivalence relation on turns in $\Delta$
by saying two directions of $\Delta$ are equivalent if they are mapped to the same direction by $f$.

\paragraph{Optimal maps.}
A piecewise linear $F$-equivariant map $f\colon \Gamma \to \Gamma'$
is an \emph{optimal map} if $\lip(f) = \lip(\Gamma,\Gamma')$ and its tension forest $\Delta(f)$
has an induced train track structure.
Meinert proves~\cite[Proposition 4.12]{Meinert}
that every $F$-equivariant piecewise linear map realizing $\lip(\Gamma,\Gamma')$
is homotopic to an optimal map $f'\colon \Gamma \to \Gamma'$
with $\Delta(f') \subset \Delta(f)$.

Meinert proves~\cite[Lemma 4.15]{Meinert}
that for an optimal map $f\colon \Gamma \to \Gamma'$,
$\lip(f) = \frac{\|g\|_{\Gamma'}}{\|g\|_\Gamma}$
if and only if the axis of $g$ in $\Gamma$ is contained in the tension forest $\Delta(f)$
and is legal with respect to the train track structure determined by $f$.

\paragraph{Lipschitz metric.}
We define the \emph{Lipschitz metric} on $\PT$ as
\[  d(\Gamma,\Gamma') = \log\lip(\Gamma,\Gamma'). \]
\begin{prop}\label{actualmetric}
    For all $\Gamma$, $\Gamma'$ and $\Gamma'' \in \PT$, we have
    \begin{enumerate}
        \item $d(\Gamma,\Gamma') \ge 0$ and $d(\Gamma,\Gamma') = 0$ implies
            $\Gamma$ and $\Gamma'$ are $F$-equivariantly isometric.
        \item $d(\Gamma,\Gamma'') \le d(\Gamma,\Gamma') + d(\Gamma',\Gamma'')$.
    \end{enumerate}
\end{prop}

\begin{proof}
    Meinert proves all of these statements as~\cite[Proposition/Definition 4.2]{Meinert}
    with the exception of the assertion that $d(\Gamma,\Gamma') = 0$ implies $\Gamma$ and $\Gamma'$
    are $F$-equivariantly isometric.
    Of course, his volume is defined slightly differently,
    but the arguments apply with our new notion of volume.

    So suppose $d(\Gamma,\Gamma') = 0$ but that $\Gamma$ and $\Gamma'$
    are not equivariantly isometric,
    and let $f\colon \Gamma \to \Gamma'$ be an optimal map.
    Because $f$ is $F$-equivariant,
    there is an induced map 
    $F\doublebackslash f\colon F\doublebackslash\Gamma \to F\doublebackslash\Gamma'$
    (it is in fact a homotopy equivalence of graphs of groups as defined in~\cite[Section 1]{Lyman}).
    Since $\Gamma$ and $\Gamma'$ are minimal,
    $f$ is surjective,
    so $F\doublebackslash f$ is as well,
    and the argument in the proof of~\cite[Lemma 4.16]{FrancavigliaMartino}~\cite[Proposition 4.16]{Meinert}
    shows that $F\doublebackslash f$ is an isometry.
    Because $f$ is not an $F$-equivariant isometry,
    it must fold a pair of edges, which must be in the same edge orbit,
    so some $\mathcal{G}'_{F\doublebackslash f(e)}$ is larger than $\mathcal{G}_e$.
    But this implies that $\vol(\Gamma') < 1$, a contradiction.
\end{proof}

\paragraph{Folding paths.}
An optimal map $f\colon \Gamma \to \Gamma'$ between trees in $\mathscr{T}$
is a \emph{folding map} if the tension forest satisfies $\Delta(f) = \Gamma$.
Associated to a folding map, one can use a technique of Skora
known as ``folding all illegal turns at speed $1$''
to obtain a $1$-parameter family of trees $\Gamma_t$ in $\mathscr{T}$ for $t \ge 0$
such that $\Gamma_0 = \Gamma$ and $\Gamma_t = \Gamma'$ for $t$ large
and folding maps $f_{st}\colon \Gamma_s \to \Gamma_t$ for $s \le t$ satisfying
\[  f_{0T} = f \text{ for all $T$ large, } f_{ss} = \operatorname{Id}_{\Gamma_s}, \text{ and }
f_{rt} = f_{st}f_{rs} \text{ for $r \le s \le t$}, \]
with $f_{rs}$ and $f_{rt}$ inducing the same train track structure on $\Gamma_{r}$.
See~\cite{BestvinaFeighnFreeFactor,FrancavigliaMartino} for more details,
in particular the argument of~\cite{BestvinaFeighnFreeFactor} that shows
that the folding paths that we are about to define are unique,
depending only on the folding map $f\colon \Gamma \to \Gamma'$.
If we rescale each $\Gamma_t$ so that it belongs to $\PT$,
Meinert shows~\cite[Theorem 4.23]{Meinert} that the
map $t \mapsto \Gamma_t$ defines a \emph{directed geodesic} in the sense that for $r \le s \le t$,
we have
\[  d(\Gamma_r,\Gamma_s) + d(\Gamma_s,\Gamma_t) = d(\Gamma_r,\Gamma_t).  \]
With our result \Cref{actualmetric} in hand,
we may reparametrize our directed geodesics by arc length,
i.e.~as a path $[0,d(\Gamma,\Gamma')] \to \PT$
such that $\Gamma = \Gamma_0$, $\Gamma_{d(\Gamma,\Gamma')} = \Gamma'$ and for $s \le t$ we have
\[  d(\Gamma_s,\Gamma_t) = t - s.  \]
This is a \emph{folding path.}

\paragraph{The Outer Space of a subgroup}
Fix $H \le F$ a subgroup of finite index $n$.
Each tree $\Gamma \in \mathscr{T}(F)$ is, by restricting the action,
a minimal $H$-tree such that all stabilizers are finite,
so we may view $\Gamma$ as a tree in $\mathscr{T}(H)$. 
Since an $F$-equivariant isometry is in particular $H$-equivariant,
we see that this map $i\colon \mathscr{T}(F) \to \mathscr{T}(H)$ is well-defined.
The natural map $H\doublebackslash\Gamma \to F\doublebackslash\Gamma$
is a covering map in the sense of Bass~\cite[2.6]{Bass}.
The equation in~\cite[Proposition 2.7]{Bass}
implies that
\[  \vol(H\doublebackslash\Gamma) = n\vol(F\doublebackslash\Gamma). \]
Thus if $\Gamma \in \PT(F)$, scaling edge lengths by $\frac{1}{n}$ yields 
a point $\frac{1}{n}\Gamma \in \PT(H)$.
We are ready to prove \Cref{isometricembeddingthm}, which we restate here.
\begin{thm}
    Suppose $H \le F$ has finite index $n$.
    The map $i\colon \PT(F) \to \PT(H)$ given by $\Gamma \mapsto \frac{1}{n}\Gamma$
    is an isometry with respect to the Lipschitz metric.
    Moreover, $i$ maps folding paths in $\PT(F)$ to folding paths in $\PT(H)$.
\end{thm}

\begin{proof}
    Fix $\Gamma$ and $\Gamma'$ in $\PT(F)$.
    If $f\colon \Gamma \to \Gamma'$ is an optimal map,
    then $\frac{1}{n}f\colon \frac{1}{n}\Gamma \to \frac{1}{n}\Gamma'$
    has Lipschitz constant $\lip(f)$, so
    \[  d_{\PT(H)}\left(\frac{1}{n}\Gamma,\frac{1}{n}\Gamma'\right) \le d_{\PT(F)}(\Gamma,\Gamma').\]
    We have seen that 
    \[  d_{\PT(H)}\left(\frac{1}{n}\Gamma,\frac{1}{n}\Gamma'\right) 
    = \log\sup_{h}\frac{\|h\|_{\frac{1}{n}\Gamma'}}{\|h\|_{\frac{1}{n}\Gamma}}
    = \log\sup_{h}\frac{\|h\|_{\Gamma'}}{\|h\|_{\Gamma}} \]
    where the supremum is taken over the hyperbolic elements of $H$.
    We have seen that there exists a hyperbolic $g_0 \in F$ realizing the supremum
    \[  \sup_g \frac{\|g\|_{\Gamma'}}{\|g\|_{\Gamma}}, \]
    where now the supremum is taken over the hyperbolic elements of $F$.
    If we take $k \ge 1$ such that $g_0^k \in H$,
    we see that
    \[  d_{\PT(H)}\left(\frac{1}{n}\Gamma,\frac{1}{n}\Gamma'\right) 
    \ge d_{\PT(F)}(\Gamma,\Gamma'), \]
    so $i$ is an isometry.
    It is clear that if $f\colon \Gamma \to \Gamma'$ is a folding map,
    then $\frac{1}{n}f\colon \frac{1}{n}\Gamma \to \frac{1}{n}\Gamma'$ is a folding map,
    and thus it follows that the folding path induced by $\frac{1}{n}f$
    agrees with that induced by $f$ and then scaling by $\frac{1}{n}$.
\end{proof}

\section{Candidates}~\label{candidatessection}
The purpose of this section is to prove \Cref{candidatesexist}.
Given trees $\Gamma$ and $\Gamma'$ in $\mathscr{T}$,
let $f\colon \Gamma \to \Gamma'$ be an optimal map.
It may not be the case that $f$ maps vertices to vertices.
We may remedy this by ($F$-equivariantly) declaring the images of vertices to be vertices.
This done,~\cite[Proposition 1.2]{Lyman}
allows us to consider the induced map of quotient graphs of groups
$F\doublebackslash f\colon F\doublebackslash\Gamma \to F\doublebackslash\Gamma'$.
The tension forest $\Delta(f)$ projects to a \emph{tension subgraph of groups}
$\Delta(F\doublebackslash f)$ in $F\doublebackslash\Gamma$.
Note that $F$-equivariance of $f$ implies that it is a subgraph of groups
and not a subgraph of subgroups in the sense of Bass~\cite{Bass}.
(Put another way, being optimally stretched is a property of the edge orbit.)
Since $f$ is optimal, $\Delta(F\doublebackslash f)$ is a \emph{core} subgraph of groups
in the sense that each valence-one vertex of $\Delta(F\doublebackslash f)$
satisfies that the edge-to-vertex group inclusion of the incident edge is not surjective.

If $g$ is hyperbolic in $\Gamma$,
the projection of the axis of $g$ to $F\doublebackslash\Gamma$,
after subdividing at preimages of vertices and modding out by the action of $g$
yields an immersion (in the sense of Bass~\cite[1.2]{Bass}) 
of a subdivided circle $C = C_g$ into $F\doublebackslash\Gamma$.
The image of $C$ in $F\doublebackslash\Gamma$
is a cyclically reduced graph-of-groups \emph{edge path}
\[  \gamma = g_0e_1g_1\ldots e_k g_k  \]
in $F\doublebackslash\Gamma$.
(See~\cite{Trees,Bass,Lyman} for more details on graphs of groups.)
In what follows we will confuse $g$, $C_g$ and $\gamma$.
We think of the edges of $C$ as being \emph{labeled} by the edges
$e_i$ of $F\doublebackslash\Gamma$ crossed by this edge path
and vertices of $C$ as being labeled by the corresponding vertex group elements.
(Note that up to cyclically re-ordering, we may assume that $g_0$ is trivial;
or put another way, one vertex is labeled $g_k g_0$.)
If $g$ is a \emph{witness,} in the sense that
\[  \lip(\Gamma,\Gamma') = \frac{\|g\|_{\Gamma'}}{\|g\|_\Gamma},  \]
then in fact, the image of $C_g$ lies in $\Delta(F\doublebackslash f)$,
and the map $C_g \to \Delta(F\doublebackslash f) \to F\doublebackslash\Gamma'$
becomes an immersion in the sense of Bass~\cite[1.2]{Bass} after subdividing at the preimages of vertices.
Thus it will sometimes instead be convenient to perform that subdivision
and think of $C_g$ as being labeled 
by edges and vertex group elements in $F\doublebackslash\Gamma'$ instead.

This done, the main step of the proof follows nearly exactly as in
Francaviglia and Martino~\cite{FrancavigliaMartino}.

\begin{lem}[``Sausages Lemma'' cf.~Lemma 3.14 of~\cite{FrancavigliaMartino}]\label{sausagelemma}
    Let $f\colon \Gamma \to \Gamma'$ be an optimal map.
    There is a witness $g$ such that the immersed circle $C_g$ is a ``sausage'':
    the image $\gamma$ of $C_g$ is a cyclically reduced path in $F\backslash\Gamma$
    of the form $\gamma = \gamma_1\bar\gamma_2$,
    where each $\gamma_i$ is a path in $F\doublebackslash\Gamma$
    that can be parametrized by $[0,1]$ in such a way that
    $\gamma_1$ and $\gamma_2$ are embedded,
    and there exist a finite family of disjoint closed intervals $I_j \subset [0,1]$,
    each one possibly consisting of a single point,
    such that $\gamma_1(t) = \gamma_2(s)$ if and only if $t = s$
    and $t$ belongs to some $I_j$.
\end{lem}

\begin{proof}
    We follow the proof of~\cite[Lemma 3.14]{FrancavigliaMartino}.
    The first claim is that we can choose $g$ 
    such that the image of $C_g$ in $F\doublebackslash\Gamma$ has no triple points.
    That is, if we can write $\gamma = \delta_1\delta_2\delta_3$,
    where the endpoints of each $\delta_i$ map to the same point in $F\doublebackslash\Gamma$,
    then we may replace $\gamma$ (and thus $g$) with a shorter loop.
    If some $\delta_i$ maps to a cyclically reduced loop in $F\doublebackslash\Gamma'$,
    choose that $\delta_i$.
    If all $\delta_i$ map to reduced but not cyclically reduced paths in $F\doublebackslash\Gamma'$,
    then (after shuffling vertex group elements around,)
    the labellings of $\delta_i$ by their images in $F\doublebackslash\Gamma'$ may be written as
    \begin{align*}
        \delta_1 &= e_1\ldots \bar e_1 \\
        \delta_2 &= e_2\ldots \bar e_2 \\
        \delta_3 &= e_3\ldots \bar e_3,
    \end{align*}
    where, since $\gamma$ immerses into $F\doublebackslash\Gamma'$,
    we must have $e_1 \ne e_2$, so the loop $\delta_1\delta_2$
    is immersed in $F\doublebackslash\Gamma'$, and we choose it instead.
    This process shortens $\gamma$ in $F\doublebackslash\Gamma$,
    so eventually we obtain $g$ such that $C_g$ has no triple points.

    The second claim is that we can choose $g$
    such that the image of $C_g$ in $F\doublebackslash\Gamma$
    has no double points which \emph{cross} in the sense that we have
    $\gamma = \delta_1\delta_2\delta_3\delta_4$,
    where the initial endpoints of $\delta_1$ and $\delta_3$ are equal
    and the initial endpoints of $\delta_2$ and $\delta_4$ are equal.
    So suppose we do have such a configuration.
    Now if some $\delta_i\delta_{i+1}$ (with subscripts taken mod $4$)
    maps to a cyclically reduced path in $F\doublebackslash\Gamma'$,
    then we can choose that path as our new $g$.
    Otherwise, as before we have that
    (after shuffling vertex group elements around)
    the labellings of $\delta_i$ by their images in $F\doublebackslash\Gamma'$ may be written as
    \begin{align*}
        \delta_1\delta_2 &= e_1\ldots \bar e_1 \\
        \delta_2\delta_3 &= e_2 \ldots \bar e_2 \\
        \delta_3\delta_4 &= e_3\ldots\bar e_3 \\
        \delta_4\delta_1 &= e_4\ldots\bar e_4.
    \end{align*}
    Therefore we conclude that
    \[  \delta_i = e_i\ldots \bar e_{i+3}, \]
    with subscripts taken mod $4$.
    Since $\gamma$ immerses into $F\doublebackslash\Gamma'$,
    we have $e_1 \ne e_3$ and $e_2\ne e_4$,
    so the loop $\delta_1\bar\delta_3$ is immersed in $F\doublebackslash\Gamma'$,
    and we choose that as our new $g$.
    Again this process shortens $\gamma$ in $F\doublebackslash\Gamma$,
    so eventually we obtain $g$ such that $C_g$ has no crossing double points.

    Finally, we claim that we can avoid \emph{bad triangles:}
    i.e.~a configuration $\gamma = \delta_1\delta_2\delta_3\delta_4\delta_5\delta_6$,
    where $\delta_1$, $\delta_3$ and $\delta_5$ are closed paths
    and thus the other three form a ``triangle,'' which need not be embedded.
    If any of the paths $\delta_1$, $\delta_2$ or $\delta_3$ is cyclically reduced
    in $F\doublebackslash\Gamma'$, then we may take that loop as our new $g$.
    If none of these subpaths are cyclically reduced in $F\doublebackslash\Gamma'$,
    then arguing as before,
    we have that $\delta_1\delta_2\delta_3\bar\delta_2$ is an immersed closed path
    in $F\doublebackslash\Gamma$ whose image in $F\doublebackslash\Gamma'$
    is cyclically reduced.
    The only worry is that this path may be \emph{longer} than $\gamma$
    because $\delta_2$ is longer than $\delta_4\delta_5\delta_6$.
    This would imply that the path $\delta_3\delta_4\delta_5\bar\delta_4$,
    to which the same argument applies,
    is shorter than $\gamma$, so we choose it.
    Again this process shortens $\gamma$ in $F\doublebackslash\Gamma$,
    so eventually we obtain $g$ such that $C_g$ has no bad triangles.

    We now think of $C_g$ as labeled by the images of edges and vertices (with vertex group elements)
    in $F\doublebackslash\Gamma$.
    Identify vertices with the same vertex label and (oriented) edges with the same edge label.
    We claim that the resulting graph is a sausage.
    If there are no identifications or only a single pair of vertex identifications, we are done.
    By assumption there are no triple points,
    so each vertex of the resulting graph is at most $4$-valent.
    Since double points do not cross, if we have two pairs of double points in $C_g$,
    they ``nest''.

    If there is only a single pair of double points, we are done,
    as $\gamma$ is a figure-eight.
    If we suppose there are at least two pairs of double points,
    then there are two ``innermost'' pairs which map to vertices in $F\doublebackslash\Gamma$.
    Unlike in the original~\cite{FrancavigliaMartino}, if $v$, $w$ is an innermost pair,
    the path between $v$ and $w$ may be of the form $eg\bar e$, so is not embedded.
    In this case, choose the middle vertex of $eg\bar e$,
    while in other cases choose a point on the embedded path from $v$ to $w$ arbitrarily.
    Since there are two innermost pairs this yields two points on $C_g$
    and therefore two subpaths $\gamma_1$ and $\gamma_2$ from the first point to the second,,
    so that $\gamma = \gamma_1\bar\gamma_2$.
    Since we have no bad triangles and no crossing double points, both $\gamma_1$ and $\gamma_2$
    are embedded in $F\doublebackslash\Gamma$.
    This shows that the graph is an embedded sausage:
    the intervals $I_j$ correspond to subpaths (which may be single vertices)
    which have more than one preimage in $\gamma$.
    Since double points do not cross,
    these intervals appear in the same order in $\gamma_1$ and $\gamma_2$.
\end{proof}

With the Sausages Lemma in hand, we can prove \Cref{candidatesexist}.

\begin{thm}
    Let $\Gamma$ and $\Gamma'$ be trees in $\mathscr{T}$.
    There is a witness $g \in F$ which is a \emph{candidate}
    in the sense of \Cref{candidatesexist}.
\end{thm}

\begin{proof}
    Again the proof is nearly identical to~\cite[Proposition 3.15]{FrancavigliaMartino}.
    Let $f\colon \Gamma \to \Gamma'$ be an optimal map.
    We begin with the loop $\gamma = \gamma_1\bar\gamma_2$ supplied by \Cref{sausagelemma}.
    If the family of intervals $I_j$ is empty or contains at most a single interval $I$,
    then $\gamma$ is a candidate and we are done.
    
    So suppose that the family of intervals $I_j$ contains at least two intervals.
    We will replace $\gamma$ with a candidate.
    Let $[a,b]$ and $[c,d]$ be the two extremal intervals of $I_j$,
    i.e. $0\le a \le b < c \le d \le 1$,
    and if $0 < a$ or $d < 1$, there is no $I_j$
    in the interval $(0,a)$ or $(d,1)$ respectively.
    We replace the path $\gamma_2$ with the path
    that follows $\gamma_2$ if $t < b$ or $t > c$
    and follows $\gamma_1$ on the interval $[b,c]$.
    Note that $\delta_2$ embeds in $F\doublebackslash\Gamma$
    because $\gamma_1(t) = \gamma_2(s)$ if and only if $t = s$.
    Note as well that the $F\doublebackslash f$-image of $\delta_2$
    in $F\doublebackslash\Gamma'$ is reduced
    for the same reason and because the $F\doublebackslash f$-images
    of $\gamma_1$ and $\gamma_2$ are reduced.
    The loop $\gamma' = \gamma\bar\delta_2$ is a candidate.
\end{proof}

\section{Induced maps on the spine and reduced spine}\label{spinesection}
Suppose $\Gamma$ is a tree in $\mathscr{T}$.
In this section we will forget the metric on $\Gamma$,
thinking of it merely as a simplicial tree.
By $F$-equivariantly collapsing an orbit of edges,
we obtain a new tree $\Gamma'$. If $\Gamma' \in \mathscr{T}$,
then we say that $\Gamma'$ is obtained from $\Gamma$ by \emph{forest collapse.}
There is an induced \emph{collapse map} in the sense of the author in~\cite[Section 1]{Lyman}
$F\doublebackslash\Gamma \to F\doublebackslash\Gamma'$,
and in the sense of that paper the edge orbit determines an edge which is a collapsible forest.
The set of $F$-equivariant homeomorphism classes of trees in $\mathscr{T}$
is partially ordered under the relation of forest collapse,
and the \emph{spine of Outer Space for $F$}, denoted $K = K(F)$,
is the geometric realization of this poset.
That is, it is an ``ordered'' simplicial complex
whose vertices are elements of the poset,
whose edges witness the partial order,
and whose higher-dimensional simplices span vertices representing totally ordered subsets of the poset.
A minimal element of this poset is a \emph{reduced} tree,
i.e.~one for which collapsing any orbit of edges yields a tree not in $\mathscr{T}$.

The \emph{reduced spine of Outer Space} $L = L(F)$ comprises
those trees $\Gamma$ all of whose edges $\tilde e$ are \emph{surviving,}
in the sense that there is a reduced collapse $\Gamma'$ of $\Gamma$
in which $\tilde e$ is not collapsed.
For example, a tree $\Gamma$ is in $L(F_n)$
if and only if the quotient graph $F_n\doublebackslash\Gamma$ has no separating edges.
In the case where $F$ is a free product of finite groups,
$L(F)$ is all of the spine of Guirardel--Levitt's Outer Space $P\mathscr{O}(F)$~\cite{GuirardelLevitt}.
Let us remark that there is an easy characterization
of those edges that are surviving
by considering ``shelters'' in the quotient graph of groups
$F\doublebackslash\Gamma$
in the sense of Clay and Guirardel--Levitt~\cite{GuirardelLevittDeformation,Clay}.
We have no further need of shelters,
which are somewhat complicated to describe,
so we refer the interested reader to Guirardel--Levitt~\cite[before Lemma 7.4]{GuirardelLevittDeformation}
or to Clay~\cite[Section 1.3]{Clay} for details.

In parallel, given a homomorphism $\alpha\colon G \to \out(F_n)$
with finite domain,
Krsti\'c--Vogtmann and Bestvina--Feighn--Handel
consider a subcomplex of the fixed-point subcomplex $K_\alpha = K_\alpha(F_n)$ of $K(F_n)$
under the action of $\alpha(G) \le \out(F_n)$.
By results of Culler~\cite{Culler} and Zimmermann~\cite{Zimmermann},
the subcomplex $K_\alpha(F_n)$ is always nonempty,
and a tree $\Gamma$ representing a vertex of $K(F_n)$
belongs to $K_\alpha(F_n)$ when,
thinking in the quotient graph,
the quotient graph $F_n\doublebackslash\Gamma$ is equipped with
a $G$-action via $\alpha$.

We, like Krsti\'c--Vogtmann and Bestvina--Feighn--Handel,
are interested in a subcomplex of $K_\alpha(F_n)$
which we term $L_\alpha = L_\alpha(F_n)$.
A tree $\Gamma$ representing a vertex of $K_\alpha(F_n)$ belongs to $L_\alpha(F_n)$ when,
thinking in the quotient graph,
every edge $e$ of $F_n\doublebackslash\Gamma$ is \emph{essential,}
in the sense that it is not contained in every maximal $G$-invariant forest 
in $F_n\doublebackslash\Gamma$.

Given a homomorphism $\alpha\colon G \to \out(F_n)$ with finite domain,
let $E$ be the full preimage of $\alpha(G)$ in $\aut(F_n)$.
We can form the fiber product 
\[  E \times_{\alpha(G)}G = \{(\Phi,g) \in E \times G : [\Phi] = \alpha(g) \}, \]
where $[\Phi]$ denotes the outer class of $\Phi$ in $\out(F_n)$.
Then $F = E\times_{\alpha(G)}G$ is a virtually free group;
in fact we have the following commutative diagram of short exact sequences
\[  \begin{tikzcd}
    1 \ar[r] & F_n \ar[d,equals] \ar[r] & F \ar[r]\ar[d] & G \ar[r]\ar[d] & 1 \\
    1 \ar[r] & F_n \ar[r] & E \ar[r] & \alpha(G) \ar[r] & 1.
\end{tikzcd}\]
Conversely, if $F$ sits in a short exact sequence as in the top row above,
the conjugation action of $F$ on itself induces a map $G \to \out(F_n)$
which we suppose to be $\alpha$,
and $F \to \aut(F_n)$ whose image is $E$,
so the diagram above commutes.
Furthermore, Goursat's lemma coupled with the fact that the kernels of the maps
$F \to G$ and $F \to E$ meet trivially
implies that $F$ is isomorphic to the fiber product $E\times_{\alpha(G)}G$.
We are ready to state and prove \Cref{normalizerthm}.

\begin{thm}\label{precisethmC}
    Given $\alpha\colon G \to \out(F_n)$ a homomorphism with finite domain,
    we may consider the groups $F$ and $E$ above.
    The isometry $i\colon \PT(F) \to \PT(F_n)$ induces a simplicial isomorphism
    $L(F) \to L_\alpha$.
    The normalizer $N(\alpha(G))$ of $\alpha(G)$ in $\out(F_n)$ leaves $L_\alpha$ invariant.
    Let $\aut^0(E)$ be the subgroup of $\aut(E)$ consisting of those automorphisms
    that leave $F_n$ invariant, and let $\out^0(E)$ be its image in $\out(E)$.
    We have the following short exact sequence
    \[  \begin{tikzcd}
        1 \ar[r] & \alpha(G) \ar[r] & N(\alpha(G)) \ar[r] & \out^0(E) \ar[r] & 1.
    \end{tikzcd}\]
\end{thm}

\begin{proof}
    As in Culler--Vogtmann~\cite{CullerVogtmann},
    we may view vertices of the spine as corresponding to trees in $\PT(F)$,
    all of whose \emph{natural} edges have the same length.
    (That is, if a vertex of $\Gamma$ has valence two,
    so that the incident edges are in the same orbit,
    we set the lengths of these edges to \emph{half}
    the length of edges that are not incident to vertices of valence two).

    Under the isometric embedding of \Cref{isometricembeddingthm},
    if $H \le F$ is a finite-index subgroup,
    vertices of the spine $K(F)$, thought of as trees in $\PT(F)$ as above,
    are sent to vertices of the spine $K(H)$.
    It is clear that this induces a poset map from the vertex set of $K(F)$ to $K(H)$.

    In the situation of the statement,
    the action of $F$ on $\Gamma$
    descends to an action of $G$ on $F_n\doublebackslash\Gamma$
    such that the \emph{quotient graph of groups} in the sense of Bass~\cite[3.2]{Bass}
    is $F\doublebackslash\Gamma$,
    and conversely every $F_n$-tree $\Gamma$ in $K(F_n)$
    such that the quotient $F_n\doublebackslash\Gamma$ is equipped with a $G$-action via $\alpha$
    is actually an $F$-tree in $K(F)$.
    We will show that a non-surviving edge in a tree $\Gamma$ corresponding
    to a vertex of $K(F)$ corresponds to an inessential $G$-invariant forest in
    $F_n\doublebackslash\Gamma$.
    Indeed, the image in $F\doublebackslash\Gamma$ of a maximal collapsible forest in $\Gamma$
    is a (collapsible) forest in the sense of the author in~\cite[Section 2]{Lyman},
    and thus a $G$-invariant forest in $F_n\doublebackslash\Gamma$.
    An edge is non-surviving if it is contained in every forest in $F\doublebackslash\Gamma$,
    and thus in every $G$-invariant forest in $F_n\doublebackslash\Gamma$
    and conversely.
    Thus the poset map above sends $L(F)$ isomorphically to $L_\alpha$.

    If $\varphi \in \out(F_n)$ belongs to the normalizer of $\alpha(G)$,
    then for all $g \in G$ we have $\varphi\alpha(g) = \alpha(g')\varphi$ for some $g' \in G$
    and for $\Gamma \in L_\alpha$, we have
    \[  (\Gamma.\varphi).\alpha(g) = (\Gamma.\alpha(g')).\varphi = \Gamma.\varphi, \]
    so $\Gamma.\varphi$ belongs to the fixed-point set of $\alpha(G)$,
    and it is clear that all of its edges remain essential,
    so we conclude $\varphi$ leaves $L_\alpha$ invariant.

    Let $\aut^0(E)$ be the subgroup of automorphisms leaving $F_n$ invariant
    and let $\out^0(E)$ be its image in $\out(E)$.
    Similarly, write $N$ for the normalizer of $\alpha(G)$ in $\out(F_n)$
    and write $M$ for its full preimage in $\aut(F_n)$.
    Following Krsti\'c~\cite{Krstic}, we will show that $M$ is isomorphic to $\aut^0(E)$.
    The final claim in the theorem will then follow from
    the snake lemma 
    (which applies because the image of each subgroup
    in the top row is normal in the bottom row)
    applied to the following commutative diagram
    \[  \begin{tikzcd}
        1 \ar[r] & F_n \ar[r]\ar[d] & E \ar[r]\ar[d] & \alpha(G) \ar[r]\ar[d] & 1 \\
        1 \ar[r] & M \ar[r,"\cong"] & \aut^0(E) \ar[r] & 1 \ar[r] & 1,
    \end{tikzcd}\]
    which yields the following exact sequence with indicated connecting homomorphism
    \[  \begin{tikzcd}
        1 \ar[r] & 1 \ar[r] & \alpha(G) \ar[r,"\delta"] & N \ar[r] & \out^0(E) \ar[r] & 1.
    \end{tikzcd}\]

    By definition, $M$ consists of automorphisms $\Phi\colon F_n \to F_n$
    such that the outer class $\varphi$ satisfies $\varphi\alpha(G)\varphi^{-1} = \alpha(G)$.
    Since $E$ is the preimage of $\alpha(G) \le \out(F_n)$ in $\aut(F_n)$,
    for $e \in E$, we have
    that the automorphism of $F_n$ defined by $x \mapsto \Phi(e\Phi^{-1}(x)e^{-1})$ is in $E$,
    say it is equal to $x \mapsto e_\Phi xe_\Phi^{-1}$.
    We claim that the map $\Psi_\Phi$ defined as $e \mapsto e_\Phi$ is an automorphism of $E$.
    For $f \in F_n$, we have $f_\Phi$ is the automorphism 
    \[
        x \mapsto \Phi(f\Phi^{-1}(x)f^{-1}) = \Phi(f)x\Phi{(f)}^{-1},
    \]
    so the restriction of $\psi_\Phi$ to $F_n$ is $\Phi$.
    An easy calculation expressing $\Phi(ee'\Phi^{-1}(x){(ee')}^{-1})$ two ways
    shows that $\psi_\Phi$ is a homomorphism.
    Since the restriction of $\psi_\Phi$ to $F_n$ is an automorphism,
    its kernel is a finite normal subgroup of $E$, hence trivial.
    Its image contains a finite-index subgroup so has finite index.
    But $E$ has a well-defined, negative Euler characteristic,
    from which we conclude that $\psi_\Phi$ is an automorphism of $E$.
    Similarly we have $\psi_{\Phi'}\psi_\Phi = \psi_{\Phi'\Phi}$,
    so the rule $\Phi \mapsto \psi_\Phi$ defines a homomorphism $M \to \aut^0(E)$.

    Conversely, given $\Phi \in \aut^0(E)$, restriction to $F_n$
    yields a homomorphism $\aut^0(E) \to \aut(F_n)$, which we want to show is in $M$.
    Indeed, for all $e \in E$ we have
    \[
        x \mapsto \Phi|_{F_n}(e\Phi|_{F_n}^{-1}(x)e^{-1}) = \Phi(e)x\Phi{(e)}^{-1},
    \]
    which is clearly in $E$.
    By definition, we have that $\psi_{\Phi|_{F_n}}(e) = \Phi(e)$,
    and we saw already that $\psi_\Phi|_{F_n} = \Phi$,
    so the operations defined above are inverse homomorphisms
    and we conclude that $M \cong \aut^0(E)$.
\end{proof}

\bibliographystyle{alpha}
\bibliography{bib.bib}
\end{document}